\documentclass[smallextended]{svjour3}
\smartqed
\usepackage{amsfonts,amsmath,amssymb,mathtools,tikz,graphicx}
\usepackage{epstopdf}
\usepackage{algorithmic}
\usepackage{algorithm}
\newcommand{\R}{\mathbf{R}}
\newcommand{\oo}{\mathcal{O}}
\newcommand{\V}{V_+}
\newcommand{\ff}{\mathbf{f}}
\newcommand{\ffp}{\mathbf{f}_+}
\newcommand{\df}{\delta\mathbf{f}}
\newcommand{\dfp}{\delta\mathbf{f}_+}
\newcommand{\eqdef}{:=}
\newcommand{\mat}[1]{\begin{bmatrix}#1\end{bmatrix}}

\begin{document}

\title{Efficient Calculation of Regular Simplex Gradients}
\titlerunning{Regular Simplex Gradients}        

\author{Ian Coope         \and
        Rachael Tappenden }

\institute{I. Coope and R. Tappenden \at
              School of Mathematics and Statistics, University of Canterbury, Private Bag 4800, Christchurch 8041, New Zealand\\
              \email{ian.coope@canterbury.ac.nz; rachael.tappenden@canterbury.ac.nz}           
}

\date{Received: date / Accepted: date}

\maketitle

  \begin{abstract}
    Simplex gradients are an essential feature of many derivative free optimization algorithms, and can be employed, for example, as part of the process of defining a direction of search, or as part of a termination criterion. The calculation of a general simplex gradient in $\R^n$ can be computationally expensive, and often requires an overhead operation count of $\oo(n^3)$ and in some algorithms a storage overhead of $\oo(n^2)$. In this work we demonstrate that the linear algebra overhead and storage costs can be reduced, both to $\oo(n)$, when the simplex employed is regular and appropriately aligned.
    We also demonstrate that a gradient approximation that is second order accurate can be obtained cheaply from a combination of two, first order accurate (appropriately aligned) regular simplex gradients. Moreover, we show that, for an arbitrarily aligned regular simplex, the gradient can be computed in $\oo(n^2)$ operations.
  \end{abstract}


\section{Introduction}\label{S_Intro}
Estimating derivatives is important in a wide variety of applications and many successful numerical optimization algorithms rely on gradient information and/or directional derivatives. When analytical derivatives are not directly available, it is useful to be able to obtain gradient estimates, for example, by using difference methods. Furthermore, simplex gradients are often used in derivative-free optimization as search directions, like is the case of the implicit filtering algorithm \cite{Bortz98}, as descent indicators for reordering the poll directions in directional direct search \cite{Custodio07}, or in the definition of stopping criteria for algorithms \cite{Custodio08}. A first comprehensive study on the computation of general simplex gradients was provided in \cite{Regis15}. In this work we investigate computationally efficient approaches to estimating the gradient at either the centroid or vertex of an appropriately aligned regular simplex.

To obtain a first order approximation to the gradient (of the function $f$ at some point $x_0$) one considers the first order Taylor approximation about $x_0$:
  \begin{equation*}
    f(x)  = f(x_0) + (x-x_0)^T\nabla f(x_0) +\oo(\|x-x_0\|_2^2).
  \end{equation*}
 Consider a set of $m+1$ points ($m\geq n$), $x_0,x_1,\dots,x_m\in \R^n$. Using the notation $g \approx \nabla f(x_0)$ to denote an approximation to the gradient, $f_j \eqdef f(x_j)$, and ignoring the order terms, the previous expression leads to the following system of equations
  \begin{eqnarray}\label{geq}
    f_j-f_0 = (x_j-x_0)^Tg  \quad \text{ for } j=1,\dots,m.
  \end{eqnarray}
Expression \eqref{geq} is a linear regression model, and determining a least squares solution to the system \eqref{geq} results in an approximation to the gradient of the underlying function.
If $m=n+1$ and the $n+1$ points are affinely independent then \eqref{geq} is a determined system with unique solution independent of the ordering of the points.
When $m>n+1$ the order of the points used is important because the least squares solution to the system \eqref{geq} depends on which point is labelled $x_0$.

In this work we restrict our attention to the case where $m = n+2$ and the points $x_1,x_2,...,x_{n+1}$ defining the regression model in \eqref{geq} are the vertices of an appropriately aligned regular simplex and $x_0$ is its centroid. (This will be defined in Section~\ref{S_NotationPreliminaries}.) The main theme here is to determine a least squares solution to the system \eqref{geq} efficiently, both in terms of the linear algebra costs and in terms of storage requirements, to determine an appropriately aligned regular simplex gradient.

For the regular simplexes discussed in this work, the centroid of the simplex is denoted by $x_0$, and each vertex $x_1,x_2,...,x_{n+1}$ is equidistant from the centroid with
\begin{equation}\label{eq:h}
  h \eqdef \|x_j-x_0\|_2, \quad j=1,2,\dots,n+1,
\end{equation}
where the distance $h$ is sometimes referred to as the `radius' or `arm length' of the regular simplex.

The system in \eqref{geq} is central to many derivative free optimization algorithms, but solving it can be a computational challenge. Firstly, usually the vectors $x_0,\dots,x_{n+1}$ (or the differences $x_1-x_0,\dots,x_{n+1}-x_0$) must be explicitly stored, which can be costly in terms of memory requirements, and also poses a limitation in terms of the size of problems that can be solved using such algorithms. Secondly, the computational cost (number of floating point operations) of solving such problems can also be high (e.g., if the problem is unstructured or if a general simplex is used).

In this work we investigate the use of regular simplexes. The computation of regular simplex gradients was proposed in the context of derivative-free optimization of molecular geometries \cite{Alberto04}. One advantage of using a regular simplex is that it provides a uniform, economic `tiling' of $n$-dimensional space, each $n$-dimensional tile having only $n+1$ vertices compared to $2^n$ vertices for a hypercube tile. A disadvantage is that storing the vertices of the simplex is usually less efficient than that for a hypercube because it is possible to align the edges
of the hypercube with the coordinate axes.  However, if the orientation of
the regular simplex is free to be chosen also, then we will show that it is possible to generate each vertex from a single vector by  simple adjustment  of one component.  This enables considerable savings in storage requirements for several lattice search algorithms for optimization. For example, the multidirectional search (MDS) method
of Torczon \cite{Torczon89}, \cite{Torczon91} can be implemented using either
a rectangular or a regular simplex based lattice but the usual construction
for the regular simplex lattice requires $\oo(n^2)$ storage (see for example \cite{Torczon89}).
Similarly, the Hooke and Jeeves \cite{Hooke+J61} lattice search method, although originally implemented in a rectangular lattice framework, can also
be implemented using a regular simplex lattice.
(It is anticipated that each of these methods will benefit, in terms of memory requirements and computational effort, if the particular simplex construction used in this work is employed.)
The added advantage of being able to compute an aligned regular simplex gradient in $\oo(n)$ housekeeping operations using only $n+1$ function evaluations makes it attractive for many numerical gradient based algorithms for optimization, including the recent minimal positive basis based direct search conjugate gradient methods described in \cite{Liu2011}.

The vertices of a simplex are often explicitly required during the initialization of simplex based algorithms for optimization, including the algorithms in
\cite{Nelder1965},
 \cite{Parkinson+H72},
 \cite{Price+CB2002},
\cite{Spendley+H62},
 \cite{Torczon89}.
    Using the technique described later, the vertices of an aligned regular simplex can be constructed explicitly, whenever required, very efficiently. However, we also show that the vertices of the aligned regular simplex do not have to be stored in order to calculate the simplex gradient.

In derivative free optimization, one must always be mindful of the cost of
function evaluations. There exist real-world applications for which computing
a single function evaluation can be very costly, and in such cases
it is clear that the linear algebra and memory requirements may be very small in comparison. In this work, we focus on algebraically efficient methods to compute the simplex gradient \emph{after function evaluations are complete}. In most situations, function evaluations will dominate the overall time to compute a simplex gradient. However, this trend should not be used to justify performing other portions of the computation inefficiently. It is prudent to be as economical as possible at every stage of the optimization process.




\subsection{Contributions}
We state the main contributions of this work (listed in order of appearance).
\begin{enumerate}
  \item \textbf{Aligned regular simplex gradient in $\mathbf{\oo(n)}$ operations.} A simplex gradient is the (least squares) solution of a system of linear equations, which, in general, comes with an associated $\oo(n^3)$ computational cost. In this work we show that, if one employs a regular simplex that is appropriately aligned, then the linear system simplifies, and the aligned regular simplex gradient can be computed in $\oo(n)$ operations. Indeed, the gradient of the aligned regular simplex is simply a weighted sum of a vector containing function values (measured at the vertices of the simplex) and a constant vector. This is an important saving over the $\oo(n^3)$ computational cost of solving a general unstructured linear system. (See Section~\ref{S_Simplex}.)
  \item \textbf{Aligned regular simplex need not be explicitly stored.} In this work we demonstrate that the storage needed for the computation of the aligned regular simplex gradient is $\oo(n)$, whereas the usual storage requirements for the computation of a general simplex gradient are at least $\oo(n)$ \emph{vectors} (i.e., $\oo(n^2)$). In particular, it is simple and inexpensive to construct the vertices of the aligned regular simplex on-the-fly, and the simplex need not be stored explicitly at all. This is because all that is required to uniquely specify (and construct) each simplex vertex is the centroid $x_0$, the simplex `arm length' $h$ \eqref{eq:h} and the problem dimension $n$. To compute the aligned regular simplex gradient, the function values at the vertices of the simplex are required, but once a vertex has been constructed and the function value found, the vertex can be discarded. Therefore, the storage requirements of this approach are low. (See Section~\ref{S_Simplex}.)
  \item \textbf{Function value $f_0$ is not required.} We show that the function value $f_0$ at the centroid of the regular simplex is not required in the calculation of the regular simplex gradient (at the point $x_0$).
      Moreover, we extend this result to show that it also applies to any general simplex, and not just a regular one. (See Section~\ref{S_Weight}.)
  \item \textbf{Regular simplex gradient in $\mathbf{\oo(n^2)}$ operations.} In some applications, it may not be possible to ensure the particular alignment of the regular simplex. In this case, we show that it is still possible to calculate the regular simplex gradient in $\oo(n^2)$ floating point operations. (See Section~\ref{S_RegSimp}.)
  \item \textbf{Inexpensive $\mathbf{\oo(h^2)}$ gradient approximation.} We show that one can efficiently compute an accurate (order $h^2$) gradient approximation using a Richardson extrapolation type approach. Specifically, if two first order accurate aligned regular simplex gradients are combined in a particular way, then a second order accurate approximation to the true gradient $\nabla f(x_0)$ is obtained. That is, an $\oo(h^2)$ gradient approximation is simply the weighted sum of two $\oo(h)$ aligned regular simplex gradients. Moreover, no additional storage is required to generate the $\oo(h^2)$ gradient approximation. (See Section~\ref{S_Oh2}.)
\end{enumerate}


%


\subsection{Paper Outline}
This paper is organised as follows. In Section~\ref{S_NotationPreliminaries} we introduce the notation and technical preliminaries that are necessary to describe and set up the problem of interest. In particular, we introduce the concepts of a minimal positive basis, how a minimal positive basis is related to a simplex, and we also state the definition of a simplex gradient. In Section~\ref{S_Simplex} the main results of this work are presented, including how to construct the simplex and how to compute the aligned regular simplex gradient in $\oo(n)$ operations. In Section~\ref{S_extensions} we describe several extensions to the work presented in Sections~\ref{S_Intro},~\ref{S_NotationPreliminaries}~and~\ref{S_Simplex}, including a special case of a regular simplex with integer entries, as well as a technique to obtain an $\oo(h^2)$ gradient approximation from two $\oo(h)$ aligned regular simplex gradients. Numerical experiments are presented in Section~\ref{S_Numerical} to demonstrate how the aligned regular simplex and its gradient can be computed in practice, as well as how to generate an $\oo(h^2)$ gradient approximation. Finally we make our concluding remarks in Section~\ref{S_Conclusion}, and we also discuss several ideas for possible future work.

\section{Notation and Preliminaries}\label{S_NotationPreliminaries}
Here we define the variables that are used in this work, and fix the notation. We also give several preliminary results that will be used later.

\subsection{Notation}\label{S_Notation}
Consider a set of $n+2$ points $x_0,x_1,\dots,x_{n+1} \in \R^n$, where $x_0$
is the centroid of the $n+1$ points $x_1,\dots,x_{n+1}$, and suppose that the function values $f_1,\dots,f_{n+1}$ are known. The function value $f_0$ also appears in this work, although we will present results to confirm that it is not used in the computation of the aligned regular simplex gradient, so it is unnecessary to assume that $f_0$ is known. Let $e$ be the (appropriately sized) vector of all ones and define the following vectors,
  \begin{equation}\label{fvec}
    \ff \eqdef \mat{f_1\\\vdots\\f_n},\qquad\text{and}\qquad \df \eqdef \ff-f_0 e = \mat{f_1-f_0\\\vdots\\f_{n}-f_0},
  \end{equation}
  along with their `extended' versions,
  \begin{equation}\label{fvecex}
    \ffp = \mat{\ff\\f_{n+1}},\qquad\text{and}\qquad \dfp \eqdef \ffp-f_0 e = \mat{\df\\f_{n+1}-f_0} .
  \end{equation}

  For a general simplex (to be defined precisely in the next section), the internal `arms' of the simplex are $\nu_j = x_j-x_0$ for $j=1,\dots,n+1.$ However, in this paper, we will only be considering \emph{regular} simplexes. In this case it is convenient to denote the `arms' of the regular simplex using the vectors $v_1,\dots,v_{n+1}$, which satisfy the relationship
  \begin{eqnarray}\label{vj}
    x_j=x_0 +hv_j,  \qquad \text{ for }\; j=1,\dots,n+1,
  \end{eqnarray}
  for some (fixed) $h \in \R$, with $\|v_j\|_2 = 1$ and $\|x_j-x_0\|_2=h$ for $j=1,\dots,n+1$.
  Thus $h>0$ is the radius of the circumscribing hypersphere of the regular simplex and each $v_j$ denotes a \emph{unit} vector defining the direction
  of each vertex from the centroid of the simplex.

Now we can define the matrix
  \begin{eqnarray}\label{GenericV}
    V = \mat{v_1&\dots&v_n}\in \R^{n\times n}
  \end{eqnarray}
  and the vector
  \begin{eqnarray}\label{vnp1}
    v_{n+1} = -\sum_{j=1}^n v_j \equiv -Ve,
  \end{eqnarray}
along with the `extended' matrix
  \begin{eqnarray}\label{Vplus}
  \V \eqdef \mat{V& -Ve} \in \R^{n\times (n+1)}.
  \end{eqnarray}

\subsection{Technical Preliminaries}\label{S_Preliminaries}

Here we outline several technical preliminaries that will be used in this work. These properties are known but we state them here for completeness. For further details on the results discussed here, see, for example,
\cite{Audet17},
\cite[p.32--34]{Boyd04},
 \cite[Chapter~2]{Conn09}.

\begin{definition}[Affine independence, pg.~29 in \cite{Conn09}]
  A set of $m+1$ points $y_1,y_2\dots,y_{m+1}\in\R^n$ is called affinely independent if the vectors\\ $y_2-y_1,\dots,y_{m+1}-y_1$ are linearly independent.
\end{definition}

\begin{definition}[Definition~2.15 in \cite{Conn09}]
  Given an affinely independent set of points $\{y_1,\dots,y_{m+1}\}$, its convex hull is called a simplex of dimension $m$.
\end{definition}

\begin{definition}
  A regular simplex is a simplex that is also a regular polytope.
\end{definition}

A regular simplex has many interesting properties, see for example \cite{ElGebeily04}.
\begin{proposition}\label{Prop:1}
A regular simplex satisfies the following properties.
\begin{enumerate}
  \item The distance between any two vertices of the simplex is constant.
  \item The centroid of a regular simplex is equidistant from each vertex.
  \item The angle between the vectors formed by joining the centroid to any two vertices of the simplex is constant.
\end{enumerate}
\end{proposition}
\begin{proof}
  The first property is a direct consequence of the definition. The second property is established in Theorem~10 in \cite{ElGebeily04}. The third property follows from the first and second properties.
\end{proof}
Thus, for a regular simplex, using Proposition~\ref{Prop:1} it can be established (see for example, \cite{ElGebeily04}) that the centroid of the simplex $x_0$ is equidistant from each vertex of the simplex, and we will say that each (internal) simplex `arm' (vectors $v_j$ for $j=1,\dots,n+1$) is of equal length, and the angles between any two arms of the simplex are equal.

The positive  span of a set of vectors $\{y_1,\dots,y_m\}$ in $\R^n$ is the convex cone
\begin{equation*}
  \{y \in \R^n : y = \alpha_1 y_1 + \dots + \alpha_m y_m, \; \alpha_i \geq 0, i = 1,\dots, m\}.
\end{equation*}

\begin{definition}[Definition~2.1 in \cite{Conn09}]\label{Def_posspanset}
  A positive spanning set in $\R^n$ is a set of vectors whose positive span is $\R^n$. The set $\{y_1,\dots,y_m\}$ is said to be positively dependent if one of the vectors is in the convex cone positively spanned by the remaining vectors, i.e., if one of the vectors is a positive combination of the others; otherwise, the set is positively independent. A positive basis in $\R^n$ is a positively independent set whose positive span is $\R^n$.
\end{definition}

\begin{remark}
  Definition~\ref{Def_posspanset} is taken directly from \cite[Definition~2.1]{Conn09}. As is stated in Footnote 2 of that work, ``strictly speaking we should have written \emph{nonnegative} instead of positive, but we decided to follow the notation in \cite{Davis54,Lewis99}''.
\end{remark}

\begin{lemma}[Minimal Positive Basis, Corollary~2.5 in \cite{Conn09}]\label{L_minposbasis}
\quad
  \begin{itemize}
  \item[(i)] $\mat{I&-e}$ is a minimal positive basis.
\item[(ii)] Let $W=\mat{w_1&\dots&w_n}\in \R^{n\times n}$ be a nonsingular matrix. Then $\mat{W&-We}$ is a minimal positive basis for $\R^n$.
\end{itemize}
\end{lemma}

Proving the existence of a regular simplex in $\R^n$ is equivalent to proving the existence of a minimal positive basis with uniform angles in $\R^n$, which is established in \cite{Alberto04}. Moreover, the work \cite{Lazebnik} establishes the existence of a regular simplex by an induction argument.

In this work we are considering the set-up where $x_0$ is the centroid of the regular simplex in $\R^n$ with vertices $x_1,\dots,x_{n+1}$. The arms of the simplex $v_1,\dots,v_{n+1}$ (defined in \eqref{vj}) form a minimal positive basis. (This will be discussed in more detail in the sections that follow.) To make this more concrete, Figure~\ref{Fig_Simplex} shows a regular simplex in $\R^2$.
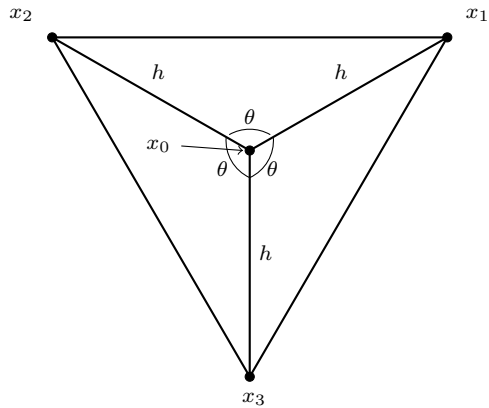
\begin{figure}[h!]\centering
    \begin{tikzpicture}
    [scale=3]
  \draw[thick] (0,0) --(0,-1);
  \draw[thick] (0,0) --(0.8660,0.5);
  \draw[thick] (0,0) --(-0.8660,0.5);
  \draw[thick] (0,-1) -- (0.8660,0.5) -- (-0.8660,0.5) -- (0,-1);
  \draw [fill] (0,0) circle [radius=0.02];
  \draw [fill] (0,-1) circle [radius=0.02];
  \draw [fill] (0.8660,0.5) circle [radius=0.02];
  \draw [fill] (-0.8660,0.5) circle [radius=0.02];
  \draw[->] (-0.3,0.02) -- (-0.03,0);
  \node at (-0.4,0.02) {$x_0$};
  \node at (1.0,0.6) {$x_1$};
  \node at (-1.0,0.6) {$x_2$};
  \node at (0.02,-1.10) {$x_3$};
  \node at (0.07,-0.45) {$h$};
  \node at (0.4,0.35) {$h$};
  \node at (-0.4,0.35) {$h$};
  \draw (0,-0.12) arc (-65:5:0.18);
  \draw (0,-0.12) arc (-115:-185:0.18);
  \draw (-0.09,0.07) arc (-240:-300:0.18);
  \node at (0.1,-0.08) {{\small$\theta$}};
  \node at (-0.12,-0.08) {{\small$\theta$}};
  \node at (0,0.15) {{\small$\theta$}};
\end{tikzpicture}
\caption{A regular simplex in $\R^2$ generated by a minimal positive basis with uniform angles. Points (vertices) $x_1,x_2,x_3$ are affinely independent and their convex hull is the regular simplex, while $x_0$ is the centroid. Each `arm' of the simplex $x_j-x_0$, for $j=1,2,3$ has the same length $h$, and the angle between any two arms is equal.}
\label{Fig_Simplex}
\end{figure}

\subsection{Simplex Gradients}\label{S_simpgrad}

The following defines a simplex gradient.
\begin{definition}[Simplex gradient, Section~2.6 in \cite{Conn09} and its generalization \cite{Regis15}]\label{Def_SimplexGradient}
  When there are $n+2$ (or more) points, $y_1,\dots,y_m\in\R^n$ with $m\geq n+2$, containing a proper subset of affinely independent points, the simplex gradient is defined as the least-squares solution of the linear system
  \begin{equation}\label{newlabel}
    f(y_j) - f(y_{1}) = (y_j-y_1)^Tg, \quad \text{for } j = 2,\dots,m.
  \end{equation}
\end{definition}
This definition depends upon whichever point is labeled $y_1$ and as a consequence, it is sometimes referred to as the simplex gradient \emph{at the point} $y_1$.

In this paper we consider only $n+1$ or $n+2$ points (either with or without the centroid $x_0$). In the case $m=n+1$ the aligned regular simplex gradient is independent of the ordering of the points because the system generated by these points is a determined system. In the case $m=n+2$, $y_1$ is the centroid of the other $n+1$ points and it can be shown (see Section~\ref{S_alternate}) that the system generated by these points is equivalent to a determined system.

Using the results in Sections~\ref{S_Notation}~and~\ref{S_Preliminaries}, \eqref{newlabel} can be rewritten in matrix notation as
  \begin{eqnarray}\label{E_simplexsystem}
    \V^T g = \tfrac1h \dfp.
  \end{eqnarray}

Definition~\ref{Def_SimplexGradient} makes clear that, in the setup used in this work with the $n+2$ points $x_0,\dots,x_{n+1}\in\R^n$, $g$ satisfies the normal equations form of \eqref{E_simplexsystem}:
  \begin{eqnarray}\label{E_SimplexGradientNormalEqns}
    \V\V^T g = \tfrac1h \V\dfp.
  \end{eqnarray}
  For further discussion on simplex gradients in a more general setting, see for example \cite{Conn09}, \cite{Custodio08}, \cite{Regis15}.

\section{Constructing the simplex}\label{S_Simplex}
The central goal of this work is to determine a least squares solution to the system \eqref{E_SimplexGradientNormalEqns} in $\oo(n)$ operations/computations, while maintaining $\oo(n)$ storage for the aligned regular simplex. One cannot hope to achieve this for a \emph{generic} simplex. However, if one can choose the simplex to be a \emph{regular simplex} that is \emph{oriented} in a particular way, then this goal can be achieved. This section is devoted to the construction of an aligned regular simplex that can be stored in $\oo(n)$ and whose gradient can be evaluated in $\oo(n)$ operations.

\subsection{Positive basis with uniform angles}

Several properties of a positive basis with uniform angles are stated now. The description uses several of the concepts already presented in \cite[Chapter~2]{Conn09}.

Consider $n+1$ \emph{normalized} vectors $v_1,\dots,v_{n+1}\in\R^n$, where the angle $\theta$ between any pair of vectors $v_i,v_j$, for $i\neq j$ is equal. It can be shown that (see \cite[Exercise~2.7(4)]{Conn09})
  \begin{eqnarray}\label{vjangle}
     \cos{\theta} = v_i^Tv_j = -\frac1n, \qquad i,j\in\{1,\dots,n+1\},\;\;i\neq j.
  \end{eqnarray}

If \eqref{vj}, \eqref{vnp1} and \eqref{vjangle} hold, then $x_1,\dots,x_{n+1}\in \R^n$ are the vertices of a regular simplex with centroid $x_0$. Thus, we seek to construct a positive basis of $n+1$ normalized vectors $v_1,\dots,v_{n+1}\in \R^n$ such that properties \eqref{vnp1} and \eqref{vjangle} hold.  With \eqref{vjangle} in mind, first the aim is to find a matrix $V$ satisfying (see (2.2) in \cite{Conn09})
  \begin{eqnarray}\label{A}
    A = V^TV= \begin{bmatrix}
      1    & -\frac1n  & \dots &-\frac1n\\
      -\frac1n & 1 &  &\vdots\\
      \vdots &  & \ddots & -\frac1n\\
      -\frac1n & \dots   & -\frac1n & 1\\
    \end{bmatrix}.
  \end{eqnarray}
  From \eqref{A}, one may write
  \begin{eqnarray}\label{Aalphabeta}
    A = V^TV= \left(1+\tfrac1n\right) I - \tfrac1n ee^T= \alpha^2 (I - \beta ee^T),
  \end{eqnarray}
  where
  \begin{eqnarray}\label{alphabeta}
    \alpha\eqdef \sqrt{\frac{n+1}{n}} \qquad \text{and} \qquad \beta \eqdef \frac1{n+1}.
  \end{eqnarray}
    Using \eqref{vnp1} and \eqref{A}, a positive basis with uniform angles exists. In particular, $A$ in \eqref{A} is symmetric and positive definite (see, for example, \cite[pg.20]{Conn09}, \cite{Golub16}) so it has a Cholesky decomposition $A=R^TR$. Taking $V=R$, which is nonsingular, combined with \eqref{vnp1} and applying Lemma~\ref{L_minposbasis}, establishes that $R_+=\mat{R&-Re}$ is a normalized minimal positive basis with uniform angles, as pointed out in \cite[p.20]{Conn09}.
The particular structure of $A$ allows the Cholesky  factor $R$ to be calculated
efficiently.

 There is, however, another factorization of $A$ that comes from the fact that any symmetric positive definite matrix has a \emph{unique} symmetric positive definite square root \cite[p.149]{Golub16}.  We search for a square-root matrix with similar structure to $A$. In particular, let
\begin{eqnarray}\label{V}
    V=\alpha(I-\gamma ee^T),
  \end{eqnarray}
where we must now specify $\gamma\in \R$. Since $A=V^TV=V^2$ it is clear that $\gamma\in \R$ must satisfy
\begin{eqnarray*}
    I-\beta ee^T=(I-\gamma ee^T)^2=I-2\gamma ee^T +n\gamma^2ee^T.
  \end{eqnarray*}
  Equating the coefficients of $ee^T$ one sees that $\gamma$ is a root of
  the quadratic equation
  \begin{eqnarray}\label{E_Quadratic}
    n\gamma^2 -2\gamma + \beta =0,
  \end{eqnarray}
giving two possible solutions:

\begin{eqnarray}\label{gamma}
    \gamma = \frac1n\left(1\pm \frac1{\sqrt{n+1}}\right).
\end{eqnarray}
Letting $\gamma_1,\gamma_2$ denote these two solutions and $V_1,V_2$ the corresponding matrices defined in \eqref{V} it is easy to show that
$V_1=HV_2$ where $H=I-\tfrac2nee^T$ is an elementary Householder reflection
matrix ($V_2$ is the reflection of $V_1$ in the hyperplane through the origin
with normal vector $e$ and vice-versa).

Choosing the negative sign for $\gamma$ in \eqref{gamma} yields the unique positive definite square-root matrix as the following lemma shows.

\begin{lemma}\label{eigV}
    Let $\alpha$, $\beta$ and $\gamma$ be defined in \eqref{alphabeta} and \eqref{gamma}. The matrix $V = \alpha(I - \gamma ee^T)$ is nonsingular. Moreover, $V$ has $n-1$ eigenvalues equal to $\alpha$ and one eigenvalue satisfying
    \begin{equation}
      \lambda_n(V) =
      \begin{cases}
        \tfrac{1}{\sqrt{n}}  & \text{ if }\gamma = \frac1n(1-\sqrt{\beta}),\\
        -\tfrac{1}{\sqrt{n}} & \text{ if } \gamma = \frac1n(1+\sqrt{\beta}).
      \end{cases}
    \end{equation}
  \end{lemma}
  \begin{proof}
    The matrix $-\alpha\gamma ee^T$ has $n-1$ zero eigenvalues, and one eigenvalue equal to $-\alpha\gamma n$. Further, adding $\alpha I$ to $-\alpha\gamma ee^T$ simply shifts the spectrum by $\alpha$. Therefore, $V$ has $n-1$ eigenvalues equal to $\alpha$, and the remaining eigenvalue is $\alpha(1-\gamma n) \overset{\eqref{gamma}}{=} \alpha(1-(1\pm\sqrt{\beta}))= \mp \alpha\sqrt{\beta} = \mp1/\sqrt{n}.$ Finally, all the eigenvalues are nonzero, so $V$ is nonsingular.
  \end{proof}
  \begin{corollary}
    If $\gamma = \frac1n(1-\sqrt{\beta})$ then $V$ is positive definite.
  \end{corollary}

\begin{lemma}\label{L_alignedwithe}
Let $\alpha$, $\beta$ and $\gamma$ be defined in \eqref{alphabeta} and \eqref{gamma} and let $V$ be defined in \eqref{V}. Then
\begin{eqnarray}\label{Ve}
  Ve =
  \begin{cases}
    \tfrac1{\sqrt{n}}e & \text{if } \gamma = \frac1n(1-\sqrt{\beta})\\
    -\tfrac1{\sqrt{n}}e & \text{if } \gamma = \frac1n(1+\sqrt{\beta}).
  \end{cases}
\end{eqnarray}
Moreover,
\begin{eqnarray}\label{VeeV}
  Vee^TV^T = \tfrac1n ee^T.
\end{eqnarray}
\end{lemma}
\begin{proof}
  With some abuse of notation, for $\gamma = \frac1n(1\pm \sqrt{\beta})$ we have
  \begin{eqnarray*}
    Ve &\overset{\eqref{V}}{=}& \alpha(I-\gamma ee^T)e =\alpha(1 - n\gamma) e =  \mp\tfrac1{\sqrt{n}} e,
  \end{eqnarray*}
  which proves \eqref{Ve}. The result \eqref{VeeV} follows immediately.
\end{proof}

Now we present the main result of this subsection, which shows that the choice $V$ in \eqref{V} leads to a minimal positive basis with uniform angles.

\begin{theorem}\label{T_minposbasis}
  Let $\alpha$, $\beta$, $\gamma$ and $V$ be defined in \eqref{alphabeta}, \eqref{gamma} and \eqref{V} respectively. Then $\V=\mat{V&-Ve}$ is a minimal positive basis with uniform angles.
\end{theorem}
\begin{proof}
  By Lemma~\ref{eigV}, $V$ is nonsingular, so applying Lemma~\ref{L_minposbasis} shows that $\V$ is a minimal positive basis.

  It remains to show the uniform angles property. By construction, $V$ defined in \eqref{V} satisfies \eqref{A}. Then
  \begin{eqnarray*}
    \V^T\V = \mat{V^T\\-(Ve)^T}\mat{V&-Ve} = \mat{V^2&-V^2e\\ -(V^2e)^T & e^TV^2e}\in \R^{(n+1)\times (n+1)}.
  \end{eqnarray*}
  Furthermore, by \eqref{Ve},
\begin{eqnarray*}
   V^2e = V(Ve) =  V\Big(\frac1{\sqrt{n}}e\Big)
   =\frac1n e,
  \end{eqnarray*}
   and $e^TV^2e = e^Te/n = 1$ so that
   \begin{eqnarray}
     \V^T\V= \begin{bmatrix}
      1    & -\frac1n  & \dots &-\frac1n\\
      -\frac1n & 1 &  &\vdots\\
      \vdots &  & \ddots & -\frac1n\\
      -\frac1n & \dots   & -\frac1n & 1\\
    \end{bmatrix}\in \R^{(n+1)\times (n+1)}.
   \end{eqnarray}
Hence, $v_1,\dots,v_{n+1}$ also satisfy \eqref{vjangle}, so the positive basis has uniform angles.
\end{proof}
Although not explicitly stated, the positive basis derived from \eqref{V} has essentially been used (with a scaling factor and origin shift), for setting up initial regular simplexes by several authors (\cite{Dennis+T91}, \cite[p. 267]{Belegundu99}, \cite[p. 80]{Jacoby74} )
\begin{remark}
  Lemma~\ref{L_alignedwithe} and Theorem~\ref{T_minposbasis} explain why the terminology `aligned regular simplex' is used in this work. Theorem~\ref{T_minposbasis} shows that $\V$ is a minimal positive basis with uniform angles, so the resulting simplex is regular. Moreover, Lemma~\ref{L_alignedwithe} demonstrates that $Ve$, which is an `arm' of the simplex (recall Figure~\ref{Fig_Simplex}), is always proportional to $e$; one arm of the regular simplex is always aligned with the vector of all ones. Finally, the choice of $\gamma$ simply dictates whether the simplex arm is oriented in the `$+e$' or `$-e$' direction.
\end{remark}

\subsection{Weight attached to centroid}\label{S_Weight}

Here we present a general result regarding the weight attached to the centroid  when solving the normal equations defining a least squares solution in linear regression. It is is well known to linear regression analysts in statistics that a  linear (affine) function, fitted by least squares, passes through the centroid of the data points.  Adding an extra `observation' at the centroid does not affect the solution for the normal of the fitted affine function --- it does, of course, affect the offset. This is irrespective of the number of data points
but has important consequences for calculating the simplex gradient at the centroid when
fitting an affine function to $n+2$ data points in $\R^n$. The following result generalises to any least squares system with $p>n$ data points ($V$ need not be a normalized invertible matrix), however, we avoid introducing extra notation by focusing on the result relating to simplex gradients.

In order to define a general simplex the following equations are used:
\begin{equation}\label{eq:nuj}
   \nu_{n+1} = -\sum_{j=1}^n \nu_j,\quad \text{where} \quad \nu_j=x_j-x_0\;\;\text{for}\;\;j=1,\dots,n+1.
\end{equation}
The vertices of the simplex are $\{x_i,\quad i=1,\dots,n+1\}$ and its centroid is $x_0$.
Here, it is \emph{not} assumed that $\|\nu_j\|_2 =1$ for all $j$, so the simplex is \emph{not necessarily} a regular simplex (i.e., \eqref{vj} need not hold).

\begin{theorem}\label{T_Weightless}
Let $\dfp$ and $\df$ be defined in \eqref{fvec} and \eqref{fvecex},  respectively, where $f_0,\dots,f_{n+1}$ are the function values at the points $x_0,\dots,x_{n+1}$. Let $V$ and $\V$ be structured as in \eqref{GenericV} and \eqref{Vplus}, respectively, but using the points $\nu_1,\dots,\nu_{n+1}$ defined in \eqref{eq:nuj}. Then the simplex gradient $g$ in \eqref{E_SimplexGradientNormalEqns} is independent of $f_0$.
\end{theorem}
\begin{proof}
 Clearly, the term $(\V\V^T)^{-1}$ in \eqref{E_SimplexGradientNormalEqns} does not involve $f_0$. Now,
 \begin{eqnarray}\label{eq:VpvsV}
  \V\dfp &=& V\df -(f_{n+1}-f_0) Ve\notag\\
  &=& V\ff - f_0Ve -f_{n+1}Ve + f_0Ve\notag\\
  &=& V(\ff -f_{n+1}e).
\end{eqnarray}
\qed
\end{proof}

Theorem~\ref{T_Weightless} shows that, if the relationship \eqref{vnp1} holds (equivalently, the summation property in \eqref{eq:nuj}), and $\V$ is a minimal positive basis, then the function value at the centroid $x_0$ is not used when computing the simplex gradient. That is, the weight attached to $x_0$ is zero when calculating a simplex gradient.

\subsection{Aligned regular simplex gradient}

Here we state and prove the main result of this work, that the aligned regular simplex gradient can be computed in $\oo(n)$ operations. We begin with the following result.

\begin{lemma}\label{L_VplusVplusT}
  Let $\alpha$, $\beta$ and $\gamma$ be defined in \eqref{alphabeta} and \eqref{gamma} and let $V$ be defined in \eqref{V}. Then, for $\V$ defined in \eqref{Vplus},
  \begin{eqnarray}\label{E_VpVp}
    \V\V^T = \alpha^2 I.
  \end{eqnarray}
  \end{lemma}
  \begin{proof}
  Note that
  \begin{eqnarray*}
    \V\V^T &=& \mat{V & -Ve}\mat{V^T\\ -(Ve)^T}\\
    &=& VV^T + Vee^TV^T\\
    &\overset{\eqref{VeeV}}{=}& V^2 + \tfrac1n ee^T\\
    &\overset{\eqref{Aalphabeta}}{=}& \alpha^2(I - \beta ee^T) + \tfrac1n ee^T\\
    &=& \alpha^2I - (\alpha^2\beta - \tfrac1n)ee^T\\
    &\overset{{\eqref{E_Quadratic}}}{=}& \alpha^2 I.
  \end{eqnarray*}
  \qed
  \end{proof}
Our main result follows, which shows that the aligned regular simplex gradient can be computed in $\oo(n)$ operations.
  \begin{theorem}\label{T_Ongrad}
    Let $\alpha$, $\beta$ and $\gamma$ be defined in \eqref{alphabeta} and \eqref{gamma}, respectively, let $V$ and $\V$ be defined in \eqref{V} and \eqref{Vplus} respectively, and let
    \begin{eqnarray}\label{E_c1c2}
      c_1 = \frac1{h \alpha} \qquad \text{and} \qquad c_2 = c_1\left((\gamma n -1)f_{n+1} -\gamma e^T\ff\right).
    \end{eqnarray}
    Then, the aligned regular simplex gradient g is computed by
    \begin{eqnarray}\label{E_Ong}
      g = c_1 \ff + c_2 e,
    \end{eqnarray}
    which is an $\oo(n)$ computation.
  \end{theorem}
  \begin{proof}
  We have
  \begin{eqnarray*}\label{g}
     g &\overset{\eqref{E_SimplexGradientNormalEqns}}{=}& \frac1{h}(\V\V^T)^{-1} \V \dfp\\
     &\overset{{\rm Lemma}~\ref{L_VplusVplusT}}{=}&\frac1{h\alpha^2} \V \dfp \\
     &\overset{\eqref{eq:VpvsV}}{=}& \frac1{h\alpha^2} V(\ff -f_{n+1}e)\\
     &\overset{\eqref{V}}{=}& \frac1{h\alpha} (I-\gamma ee^T)(\ff -f_{n+1}e)\\
     &=& \frac1{h\alpha} (\ff -f_{n+1}e -\gamma (e^T\ff)e +\gamma f_{n+1}n e)\\
     &=& \frac1{h\alpha} (\ff  +((\gamma n -1)f_{n+1} -\gamma (e^T\ff))e).
  \end{eqnarray*}
  Note that the gradient is simply the sum of two (scaled) vectors, which is an $\oo(n)$ computation (see, for example \cite[p.3]{Watkins10}). \qed
  \end{proof}

Theorem~\ref{T_Ongrad} shows that the gradient of the aligned regular simplex can be expressed very simply as a weighted sum of the function values (measured at the vertices of the simplex) and a constant vector. Thus, it is very cheap to obtain the simplex gradient once function values have been calculated.

These results also demonstrate that using this particular simplex leads to efficiencies in terms of memory requirements. Neither the vertices of the simplex $x_1\dots,x_{n+1}$, nor the arms of the simplex $v_1,\dots,v_{n+1}$, appear in the calculation of the aligned regular simplex gradient. All that is needed is the function values computed at the vertices of the simplex. Note that the vertices of the simplex need not be stored because they can be computed easily on-the-fly as follows. Recall that $V=\alpha(I-\gamma ee^T)$ \eqref{V}. Therefore, each arm of the simplex is
\begin{equation}\label{vjcheap}
  v_j = \alpha(e_j - \gamma e),
\end{equation}
where $e_j$ is the $j$th column of $I$. The $j$th vertex of the simplex is recovered via
\begin{equation}\label{xjcheap}
  x_j \overset{\eqref{vj}}{=} x_0 + h v_j \overset{\eqref{vjcheap}}{=} x_0 + h \alpha(e_j - \gamma e) = (x_0 - h\alpha \gamma e) + h\alpha e_j.
\end{equation}
Expression \eqref{xjcheap} shows that $x_j$ is simply the sum of a constant vector $(x_0 - h\alpha \gamma e)$ whose $j$th component has been modified by $h\alpha$. The only quantities necessary to uniquely determine each vertex are $x_0$, $h$ and $n$. To compute the aligned regular simplex gradient, the $j$th vertex can be generated (via \eqref{xjcheap}), the function value $f_j$ evaluated and stored in $\ff$, and subsequently, the vertex can be discarded. This confirms that the storage requirements for the aligned regular simplex gradient are $\oo(n)$.

\subsection{An alternative formulation}\label{S_alternate}

In Section \ref{S_Weight} it was shown that the weight attached to the centroid is zero so that only the function values at the vertices of the
simplex feature in the regular simplex gradient calculation. But $n+1$
affinely independent points in $\R^n$ define a unique interpolating
affine function with constant gradient and this must, therefore, coincide
with the definition of the simplex gradient defined by the $n+2$ points
used in the least-squares solution \eqref{E_SimplexGradientNormalEqns}. This means that the regular simplex gradient could also be calculated as the solution to
the square system of equations
\begin{eqnarray}\label{E_NoCentroid}
(x_j-x_{n+1})^Tg=\left(f_j-f_{n+1}\right), j=1,\dots,n.
\end{eqnarray}
It is not immediately obvious that this is an equivalent formulation.
To show this equivalence algebraically we use
the identity
$x_j-x_{n+1}=x_j-x_0 -(x_{n+1}-x_0) = h(v_j-v_{n+1})$, and the definition
of $V$ \eqref{V} and $v_{n+1}$ \eqref{vnp1}.  The linear system of equations (\ref{E_NoCentroid}) can then be rewritten
\begin{eqnarray*}    
h(v_j-v_{n+1})^Tg=\left(f_j-f_{n+1}\right), j=1,\dots,n.
\end{eqnarray*}
or in matrix form (after dividing by $h$),
\begin{eqnarray*}
(V+Vee^T)^Tg  = (V+ee^TV)g = \tfrac{1}{h}(\ff-f_{n+1}e).
\end{eqnarray*}
Premultiplying by the invertible matrix $V$ then gives
\begin{eqnarray}\label{E_NoCentroidMV}
(V^2 +Vee^TV)g &=& \tfrac{1}{h}V(\ff-f_{n+1}e).
\end{eqnarray}
Lemma~\ref{L_VplusVplusT} showed that $(V^2+Vee^TV) =\alpha^2I$, and it is then clear that solving equation (\ref{E_NoCentroidMV}) is equivalent to  finding the solution of the normal equations \eqref{E_SimplexGradientNormalEqns} by the method described in the previous section.

\begin{remark}
  We remark that a \emph{linear model} is being used throughout this work, so an affine function is fitted through the $n+1$ simplex vertices, and the simplex gradient is the gradient of the affine function. Furthermore, note that if the centroid $x_0$ is included in the calculation of the simplex gradient at $x_0$, then the offset of the affine function is affected, but this \emph{does not affect the gradient}, i.e., the simplex gradient at the centroid is the same as the simplex gradient at any vertex when the centroid is not included. (If the simplex gradient is calculated at $x_j , j\ne 0$, using the $n+2$ points then the simplex gradient will be affected.) However,
  inclusion of the centroid does simplify the derivation of error bounds as is now shown

\end{remark}

\subsection{Error bounds}

Here we state explicit bounds on the error in the regular simplex gradient, compared with the analytic gradient. First we give the following result providing an error bound for the aligned regular simplex gradient at the centroid $x_0$ and follow with an extension giving an error bound at any vertex.

\begin{theorem}\label{errorx0}
  Let $x_0$ be the centroid of the aligned regular simplex with radius $h>0$  and vertices $x_j=x_0 +hv_j,\quad j=1,2,\dots,n+1.$
  Assume that $f$ is continuously differentiable in an open domain $\Omega$ containing $B(x_0;h)$ and $\nabla f$ is Lipschitz continuous in $\Omega$ with constant $L>0$.
  Then, $g$, obtained by solving the system of linear equations (\ref{E_SimplexGradientNormalEqns}), satisfies the error bound
  \begin{equation}\label{eq:genbound}
    \|\nabla f(x_0) - g\|_2 \leq \tfrac12L h\sqrt{n}.
  \end{equation}
\end{theorem}
\begin{proof}
Using the normal equations  (\ref{E_SimplexGradientNormalEqns}) defining $g$  we can write \begin{equation}\label{eq:normalext}
    V_+V_+^T\left(g-\nabla f(x_0)\right) =
         \tfrac{1}{h}V_+\left(\dfp - hV_+^T\nabla f(x_0)\right).
  \end{equation}
The integral form of the mean value theorem provides the identity \begin{eqnarray*}\label{eq:MVT}
  f_j - f_0 =
         \int_0^1(x_j-x_0)^T\nabla f\left(x_0+t(x_j-x_0)\right) dt,
                 \quad j=1,\dots,n+1.
\end{eqnarray*}
Therefore, the $j$th component of the vector in brackets on the right-hand-side of equation (\ref{eq:normalext}) is \begin{eqnarray*}\label{eq:jcomp}
  \left(\dfp - hV_+^T\nabla f(x_0)\right)_j
            &=&  f_j - f_0 -(x_j-x_0)^T\nabla f(x_0), \\
            &=& (x_j-x_o)^T\int_0^1\left( \nabla f(x_0+t(x_j-x_0)) - \nabla f(x_0)\right) dt, \\
            &\le& \|x_j - x_0\|_2 \int_0^1L\|t(x_j-x_0\|dt, \\
            &=& L\|x_j-x_0\|_2^2\int_0^1tdt,  \\
            &=& \tfrac12L h^2, \quad j=1,\dots,n+1, \end{eqnarray*} which provides the bound \begin{equation}\label{eq:rhsbnd} \| \dfp - hV_+^T\nabla f(x_0) \|_2 \le \tfrac12L h^2\sqrt{n+1}.
\end{equation}
Because $V_+V_+^T=\alpha^2I$  equation \eqref{eq:normalext} and the bound \eqref{eq:rhsbnd} lead to the inequality
\begin{equation}\label{eq:lhsbnd} \alpha^2\|g-\nabla f(x_0)\|_2 \le \tfrac12 L h \sqrt{n+1}\|V_+\|_2.
\end{equation}
By Lemma~\ref{eigV}, $\|\V\|_2 = \alpha,$ so
\begin{eqnarray*}
  \|\nabla f(x_0) - g\|_2 \leq \frac{1}{2\alpha} h L \sqrt{n+1}.
\end{eqnarray*}
The definition of $\alpha$ in \eqref{alphabeta} gives the required result.\qed
\end{proof}

An error bound at any vertex $x_j$, $j=1,\dots,n+1$, of the regular simplex is then  easily derived from the Lipschitz continuity of the gradient of $f$ and the triangle inequality.
\begin{eqnarray*}
  \|\nabla f(x_j) - g\|_2
    \leq \|\nabla f(x_j)-\nabla f(x_0)\|_2 + \|\nabla f(x_0)-g\|_2
    \leq \left(1+\tfrac12\sqrt{n}\right)L h.
\end{eqnarray*}

\section{Extensions}\label{S_extensions}

In this section we describe several extensions of the work presented so far. In particular, we show that a regular simplex gradient, where the simplex is arbitrarily oriented, can be computed in $\oo(n^2)$ operations, we show that one can easily construct a regular simplex with integer entries when $n+1$ is a perfect square, and we also show that it is computationally inexpensive to calculate an $\oo(h^2)$ approximation to the gradient using a Richardson extrapolation type approach.

\subsection{A regular simplex gradient in $\oo(n^2)$}\label{S_RegSimp}

In practice, it may not be desirable to use the oriented regular
simplexes discussed so far. However, any regular simplex is related to that particular simplex formed from the aligned positive basis $V_+$ by  a scale factor,
an orientation (orthogonal matrix), a permutation of the columns, and
a shift of origin.
In fact the permutation can be dispensed with because if $P$ is a permutaion matrix 
then
 $$(I-\gamma ee^T)P = P-\gamma ee^TP =
   P(I-\gamma P^Tee^TP)= P(I-\gamma ee^T).$$
Thus, if  $W_+ =\mat{W& -We}$  is \emph{any}  normalized minimal positive basis
with uniform angles then,
$$ W = QVP = (QP)V$$
so that $W$ is linked to $V$ by an orthogonal transformation $QP$
(and hence $W_+$ to any other normalized minimal positive basis
with uniform angles).
These observations enable any regular simplex gradient to be calculated in $\oo(n^2)$ operations.

\begin{theorem}\label{T_General}
Let $Z_+ =\mat{z_1&\dots&z_n&z_{n+1}} =\mat{Z&z_{n+1}} \in \R^{n\times (n+1)}$ be \emph{any} regular simplex with radius $h$ and centroid $z_0$ and let $f_j = f(z_j)$, $j=1,\dots,n+1$ be known function values. Further, let
\begin{equation}\label{eq_u}
  u = \tfrac{1}{\alpha^2h^2}(\ff-f_{n+1}e).
\end{equation}
 Then the simplex gradient is
\begin{eqnarray}\label{E_gradient2}
g= Zu - (e^Tu)z_0,
\end{eqnarray}
which can be calculated in $\oo(n^2)$ floating point operations.
\end{theorem}
\begin{proof}
The interpolation conditions for the simplex gradient can be written as
\begin{eqnarray}\label{E_Zinterpolation}
\left((z_j-z_0) -(z_{n+1} -z_0)\right)^Tg= f_j-f_{n+1},\quad j=1,\dots,n.
\end{eqnarray}
Let $Y_+=\mat{Y&-Ye}$ be the regular simplex
with unit radius and with centroid at the origin defined by
\begin{eqnarray}\label{E_Ysimplex}
Y= \tfrac1h(Z - z_0e^T),
\end{eqnarray}
and let $Q \in \R^{n\times n}$ be the orthogonal transformation linking
$Y_+$ to the oriented simplex $V_+ =\mat{V&-Ve}$ where
$V= \alpha(I-\gamma ee^T)$ so that
$$
Y=QV.          
$$
The square system of equations (\ref{E_Zinterpolation}) can be written in matrix form as
\begin{eqnarray}\label{E_ZinterpolationMV}
h(Y+Yee^T)^T g= \ff - f_{n+1}e.
\end{eqnarray}
Pre-multiplying by the invertible matrix $Y$ and dividing by $h$ we get
\begin{eqnarray}\label{E_g}
(YY^T+Yee^TY^T) g = \tfrac1hY\left(\ff-f_{n+1}e\right).
\end{eqnarray}
Now
\begin{eqnarray*}
 YY^T &=& QV^2Q^T\\
      &\overset{\eqref{Aalphabeta}}{=}&  \alpha^2Q(I-\beta ee^T)Q^T \\
      &=& \alpha^2(I-\beta Qee^TQ^T).
\end{eqnarray*}
But $Q=YV^{-1}$ so $Qe=YV^{-1}e$. By \eqref{Ve}, $Ve=\pm\tfrac{1}{\sqrt{n}} e$, so that $V^{-1}e=\pm\sqrt{n}e$ and we have
$$
Qee^TQ^T = nYee^TY^T .               
$$
Therefore,
$$
 YY^T=\alpha^2I -\alpha^2\beta n Yee^TY^T.
$$
Using the definitions \eqref{alphabeta}, $\alpha^2\beta n = 1$, so that
$$YY^T = \alpha^2I -Yee^TY^T.$$
Inserting this result in \eqref{E_g} we get
$$
  g=\tfrac{1}{\alpha^2 h}Y\left(\ff-f_{n+1}e\right),
$$
which is a simple matrix-vector product costing $\oo(n^2)$ flops.
In fact we do not need to calculate $Y$.  Substituting for $Y$ from
equation \eqref{E_Ysimplex} gives
\begin{eqnarray}\label{E_gradient}
g= \tfrac{1}{\alpha^2h^2}(Z-z_0e^T)(\ff-f_{n+1}e).
\end{eqnarray}
Letting $u$ be as defined in \eqref{eq_u} gives the result \eqref{E_gradient2}. Finally, note that the dominant computation in \eqref{E_gradient2} is the matrix-vector product $Zu$, which has a computational complexity of $\oo(n^2)$ (see for example, \cite[p.2]{Watkins10}).
\qed
\end{proof}

In practice, the centroid $z_0$ will often be known but even if it is not given initially, its calculation is  at most $\oo(n^2)$ flops because $z_0 = \frac{1}{n+1}\sum_{j=1}^{n+1} z_j$. If a new simplex is formed by resizing a given simplex but keeping one vertex in common then the new centroid can be easily calculated from the old centroid and the resizing parameter in $\oo(n)$ flops. Finally, we note that if $h$ is unknown it can be calculated as $h = \|z_j-z_0\|_2$ for any $j$, which is an additional cost of $\oo(n)$ flops.

\subsection{Regular simplexes with integer entries}

The results of Section \ref{S_Simplex} show that one can construct
a regular simplex with \emph{integer} coordinate vertices in $n$-space when
$n+1$ is a perfect square.  Simply let $x_0=0$ be the centroid of the simplex
so that $x_j=hv_j, j=1,\dots,n+1$ are the $n+1$ vertices. Writing
$X_+=[x_1,\dots,x_{n+1}]$, we choose $X_+$ to be proportional to
the rational matrix $\tfrac{1}{\alpha}\V$.  For example, when $n=3$, then $n+1=4$ is a perfect square, so two examples of regular simplexes in $\R^3$ with integer coordinates, corresponding to the
two choices for $\gamma$ in \eqref{gamma}, are
\begin{eqnarray*}
  X_+ &=& \mat{5&-1&-1&-3\\-1&5&-1&-3\\-1&-1&5&-3}\in \mathbf{Z}^{3\times 4}
\end{eqnarray*}
and
\begin{eqnarray*}
  X_+ &=& \mat{1&-1&-1&1\\-1&1&-1&1\\-1&-1&1&1}\in \mathbf{Z}^{3\times 4}.
\end{eqnarray*}
Sch\"{o}enberg \cite{Schoenberg37} proved that a regular $n$-simplex exists in $\R^n$ with integer coordinates in the following cases, and no others:
\begin{itemize}
  \item[(i)] $n$ is even \emph{and} $n+1$ is a square;
  \item[(ii)] $n \equiv 3 \pmod 4$;
  \item[(iii)] $n \equiv1 \pmod4$ \emph{and} $n+1$ is a sum of two squares.
\end{itemize}
In particular, the first few values  of $n$ for which
integer coordinate vertices exist are $n=1,3,7,8,9,11,15,17,19,\dots$, and do not exist for $n=2,4,5,6,10,12,13,\\14,16,18,20\dots$.

\subsection{Order $\oo(h^2)$ gradient approximation}\label{S_Oh2}

At certain stages of an optimization algorithm an accurate gradient may be required. This is the case, for example, when deciding whether to reduce the mesh/grid size in mesh/grid based optimization algorithms, or for deciding whether a gradient based stopping condition has been satisfied. In such cases, an $\oo(h)$ gradient approximation may not be sufficient, and a more accurate gradient, say an $\oo(h^2)$ gradient approximation, may be desired.

The construction proposed in this paper allows one to obtain an inexpensive aligned regular simplex gradient, which is an $\oo(h)$ approximation to the true gradient. However, it is well known in the statistics community that a Richardson's extrapolation approach can be used to increase the accuracy of an approximation or iterative method by (at least) an order of magnitude, see for example \cite{Dimov17,Richardson1927}. Indeed, using the set-up in this paper, we now demonstrate how to obtain an $\oo(h^2)$ approximation to the true gradient in $\oo(n)$ operations and storage, although extra function evaluations will be required.

The key idea behind Richardson's extrapolation is to take two approximations that are $\oo(h)$, and use these to construct an $\oo(h^2)$ approximation. To this end, fix $x_0$, let $G = \nabla^2 f(x_0)$ and choose $h_1 = \oo(h)$. Then one can form a regular simplex with centroid $x_0$ and diameter $h_1$ with the vertices and `arms' satisfying  $x_j - x_0 = h_1 v_j$ for $j=1,\dots,n+1$. Now, consider the Taylor series of $f$ about $x_0$:
\begin{eqnarray*}
  f_j &=& f_0 + (x_j - x_0)^T \nabla f(x_0) + \tfrac12(x_j - x_0)^TG(x_j - x_0) + \oo(h^3)\\
  &=& f_0 + h_1v_j^T \nabla f(x_0) + \tfrac{h_1^2}2v_j^TGv_j + \oo(h^3).\\
\end{eqnarray*}
Rearranging the above and dividing by $h_1$ gives
\begin{eqnarray}\label{eq:interim}
  v_j^T \nabla f(x_0) = \tfrac1{h_1}(f_j - f_0) - \tfrac{h_1}2v_j^TGv_j + \oo(h^2).
\end{eqnarray}
An expression of the form \eqref{eq:interim} can be written for each $j=1,\dots,n+1$. Combining the $n+1$ equations, using the notation established previously, gives
\begin{eqnarray}\label{eq:extrap_intermediatestep}
  \V^T\nabla f(x_0) = \tfrac1{h_1}\dfp - \tfrac{h_1}2{\rm diag} (\V^TG\V)e + \oo(h^2),
\end{eqnarray}
where ${\rm diag}(\V^TG\V)$ is a diagonal matrix with $({\rm diag}(\V^TG\V))_{jj} = v_j^TGv_j$. Let
\begin{equation}\label{eq:C}
  C = - \tfrac{1}2(\V\V^T)^{-1}\V({\rm diag} V^TGV)e,
\end{equation}
so that \eqref{eq:extrap_intermediatestep} becomes
\begin{eqnarray}\label{g1vsg2}
  \nabla f(x_0) 
  &=& g_1 + h_1 C + \oo(h^2),
\end{eqnarray}
where $g_1 = \tfrac1{h_1}(\V\V^T)^{-1}\V\df$. By \eqref{E_SimplexGradientNormalEqns}, $g_1$ is an $\oo(h)$ approximation to the gradient at the point $x_0$.

Now, fix the same $x_0$ and direction vectors $v_1,\dots,v_{n+1}$, and choose some $h_2 = \oo(h)$. Then, constructing a simplex of diameter $h_2$ and following the same arguments as above, we arrive at the expression
\begin{eqnarray}\label{g1vsg2g2}
  \nabla f(x_0) = g_2 + C h_2 + \oo(h^2),
\end{eqnarray}
where $C$ is defined in \eqref{eq:C}, and $g_2 = \tfrac1{h_2}(\V\V^T)^{-1}\V\df$ is an $\oo(h)$ approximation to the gradient at the point $x_0$.

Finally, multiplying \eqref{g1vsg2} by $h_2$, multiplying \eqref{g1vsg2g2} by $h_1$ and subtracting the second expression from the first, results in
\begin{eqnarray}\label{Oh2gradgeneral}
  \nabla f(x_0) = g_{12} + \oo(h^2),\qquad \text{where} \qquad g_{12} = \frac{h_2 g_1 - h_1g_2}{h_2-h_1},
\end{eqnarray}
i.e., $g_{12}$ is an order $h^2$ accurate approximation to the true gradient at $x_0$.

Moreover, if $h_2$ is chosen to be a multiple of $h_1$ (i.e., $h_2 = \eta h_1$) then
\begin{eqnarray}\label{Oh2grad}
  g_{12} = \frac{\eta h_1 g_1 - h_1g_2}{\eta h_1-h_1} = \frac{\eta}{\eta-1}g_1 - \frac{1}{\eta-1}g_2 .
\end{eqnarray}

To make the previous arguments concrete, an algorithmic description of the procedure to find an $\oo(h^2)$ approximation to the gradient is given in Algorithm~\ref{alg:h2from2h}. Briefly, the algorithm proceeds as follows. In Steps 2--3, an $\oo(h)$ aligned regular simplex gradient is formed via equation \eqref{E_Ong} (i.e., using the procedure developed previously in this work). To obtain an $\oo(h^2)$ gradient approximation, a second (related) $\oo(h)$ aligned regular simplex gradient approximation is also needed, and this is computed in Steps 4--5 of Algorithm~\ref{alg:h2from2h}. Finally, in Step 6, a weighted sum of the two $\oo(h)$ gradients is formed, resulting in an $\oo(h^2)$ regular simplex gradient approximation.
\begin{algorithm*}[h!]
\caption{Forming an $\oo(h^2)$ gradient approximation from two $\oo(h)$ gradient approximations via Richardson's extrapolation.}
\label{alg:h2from2h}
\begin{algorithmic}[1]
    \STATE \textbf{Input: } Centroid $x_0$, problem dimension $n$, scalars $h_1 \sim \oo(h)$ and $h_2 \sim \oo(h)$.
    \STATE Input or compute: $f(x_j) \overset{\eqref{xjcheap}}{=} f(x_0+h_1\alpha(e_j - \gamma e))$ for $j = 1,\dots,n+1$.
    \STATE Compute $g_1$ via \eqref{E_Ong} using $h_1$.
    \STATE Input or compute: $f(x_j) \overset{\eqref{xjcheap}}{=} f(x_0+h_2\alpha(e_j - \gamma e))$ for $j = 1,\dots,n+1$.
    \STATE Compute $g_2$ via \eqref{E_Ong} using $h_2$.
    \STATE Compute $g_{12}$ via \eqref{Oh2gradgeneral}.
\end{algorithmic}
\end{algorithm*}

\begin{remark}\label{remarkOh2}
We make the following comments.
\begin{enumerate}
  \item The $\oo(h^2)$ gradient approximation \eqref{Oh2grad} is simply a weighted sum of two $\oo(h)$ gradient approximations. The coefficients of $g_1$ and $g_2$ sum to 1.
  \item In the context of Richardson's extrapolation, the parameter $\eta$ in \eqref{Oh2grad} can be chosen to be either positive or negative, (but, to avoid division by zero, it cannot be set to 1). However, in the context of this work, $h_1$ and $h_2$ denote the radii of simplexes, so they must be positive (recall the relationship $h_2 = \eta h_1$). We stress that, computationally, there is no issue here when $\eta<0$, but we must interpret the scaling parameter $\eta$ carefully. Geometrically, if $\eta$ is a positive value, then the simplex generated using $h_2$ (see Steps~4--5 in Algorithm~\ref{alg:h2from2h}) is simply a scaled version of the original simplex defined using $h_1$ (both simplexes sharing the common centroid $x_0$). However, if $\eta$ is negative, we still use the (negative) value $h_2$ when performing the computations in Algorithm~\ref{alg:h2from2h}, but geometrically we interpret the simplex radius to be $|h_2|$, \emph{and the simplex has been rotated by } $180^{\circ}$ (again with both simplexes sharing the common centroid $x_0$). See the numerical example in Section~\ref{SS_exp2} and Figure~\ref{Fig_Simplex_numerics}.
  \item In this section the derivation proceeds by assuming that the 2 simplex gradients $g_1$ and $g_2$ are both computed at the \emph{same point} $x_0$, and thus $g_{12}$ is an $\oo(h^2)$ accuracy approximation to $\nabla f(x_0)$ (and by results previously presented in this work, $g_1$, $g_2$ and $g_{12}$ all have a computational cost of $\oo(n)$). However, the arguments in Section~\ref{S_Oh2} can be generalized to an $\oo(h^2)$ approximation to $\nabla f(x)$, for some other point $x$ say, so long as both $g_1$ and $g_2$ are $\oo(h)$ simplex gradients at the common point $x$. Of course, the computational cost of obtaining $g_1$ and $g_2$ may be higher than $\oo(n)$ for general $x$.
      \end{enumerate}
\end{remark}

\section{Numerical example}\label{S_Numerical}

Here we present two numerical examples to make the ideas of the paper concrete, to highlight the simplicity and economy of our approach, and to demonstrate how an $\oo(h^2)$ approximation to the gradient can be constructed from two $\oo(h)$ aligned regular simplex gradients. All experiments are performed on Rosenbrock's function, and MATLAB (version 2016a) is used for the calculations.

We temporarily depart from our usual notation and let $y \in \R^2$ with components $y = \mat{y_1&y_2}^T$ so that Rosenbrock's function can be written as
\begin{eqnarray}\label{E_rosenbrock}
  f(y_1,y_2) = (1-y_1)^2+100(y_2-y_1^2)^2.
\end{eqnarray}
 The gradient of \eqref{E_rosenbrock} can be expressed analytically as
\begin{eqnarray}\label{E_rosenbrock_grad}
  \nabla f(y_1,y_2) = \mat{-2(1-y_1) -400y_1(y_2-y_1^2)\\200(y_2-y_1^2)}.
\end{eqnarray}

Henceforth, we return to our usual notation.

\subsection{Inconsistent simplex gradients}\label{SS_exp1}

The purpose of this example is to highlight a situation that is not uncommon in derivative free optimization algorithms --- that of encountering an iterate where the true (analytic) gradient and the simplex gradient point in opposite directions --- and how the construction in Section~\ref{S_Oh2} can be used to determine an accurate gradient direction from which to make further progress. This situation can arise, for example, when the gradient of a function at the iterate $x^{(k)}$ is close to flat.

Indeed, this is one of the motivations for considering Rosenbrock's function, which has a valley floor with a shallow incline. To highlight the situation previously described, we have selected a test point that is very close to the `floor' of the valley of Rosenbrock's function, where a good approximation to the gradient is required to make progress. (Ultimately, descent methods do track this valley floor, so it is not unexpected that we may encounter a point of this nature.) We stress that the loss of accuracy is due to the regular simplex gradient being a first order approximation ($\oo(h)$) to the analytic gradient, and is not because of the particular construction proposed in this work.

The example proceeds as follows. Suppose one wishes to compute a regular simplex gradient at the point
\begin{equation}\label{x0}
  x_0 = \mat{1.1\\1.1^2+10^{-5}}.
\end{equation}
Note that, from \eqref{E_rosenbrock_grad}, the true gradient at the point $x_0$ is (to the accuracy displayed)
\begin{eqnarray}\label{Grad_x0}
  \nabla f(1.1,1.1^2+10^{-5}) = \mat{0.195599999999971\\0.002000000000013}.
\end{eqnarray}
The aligned regular simplex is constructed using the approach presented in Section~\ref{S_Simplex}. In particular, $n=2$ for Rosenbrock's function so that
\begin{eqnarray}\label{Numerical_abg}
  \alpha \overset{\eqref{alphabeta}}{=}  \sqrt{\frac{3}{2}} \qquad \beta\overset{\eqref{alphabeta}}{=}  \frac13 \qquad \gamma\overset{\eqref{gamma}}{=} \frac12\left(1+\frac1{\sqrt{3}}\right).
\end{eqnarray}
Then, recalling that $V = \alpha(I-\gamma ee^T)$ (see \eqref{V}) we have
\begin{eqnarray}\label{NumericalVp}
  \V = \mat{V&-Ve} = \mat{0.2588&-0.9659&0.7071\\-0.9659&0.2588&0.7071}.
\end{eqnarray}
Recall that the connection between the arms of the simplex and vertices of the simplex is given in \eqref{vj} as $x_j = x_0 + hv_j$ for $j = 1,2,3$ and for some $h\in\R$. For this experiment we choose $h_1 = 10^{-3}$ so that the three vertices of the simplex are given as the columns of
\begin{equation}\label{E_X}
  X_+ = \mat{x_1&x_2&x_3}=\mat{1.1003&1.0990&1.1007\\1.2090&1.2103&1.2107}.
\end{equation}
The aligned regular simplex gradient (at the point $x_0$) can be computed in $\oo(n)$ operations using Theorem~\ref{T_Ongrad} (which requires the function values $f_1,f_2,f_3$ computed at the points $x_1,x_2,x_3$ via \eqref{E_rosenbrock}), and is as follows:
\begin{eqnarray}\label{SimplexGradg1}
  g_{1} = \mat{-0.095750884326868\\-0.017496117072893}.
\end{eqnarray}
Notice that the regular simplex gradient is very different from the true gradient \eqref{Grad_x0}. Not only are the magnitudes of the numbers different but the regular simplex gradient \eqref{SimplexGradg1} even has the opposite sign from the true gradient.
This loss of accuracy is  inevitable for any first order numerical method used to approximate a gradient close to a stationary point and the usual remedy is to switch to a second order method.

However, using the techniques presented in this paper it is cheap to compute an aligned regular simplex gradient. So, suppose another approximation to the true gradient is constructed, again at the point $x_0$ \eqref{x0}, but using a different simplex diameter $h_2$. That is, suppose we set $h_2 = \tfrac12 h_1$ ($h_1$ and $h_2$ are of the same order) so that $\V$ remains unchanged, but the simplex vertices become:
\begin{equation}\label{E_Xp}
  X_+' = \mat{x_1'&x_2'&x_3'} = \mat{1.1001&1.0995&1.1004\\1.2095&1.2101&1.2104}.
\end{equation}
The function values $f_1',f_2',f_3'$ are computed at the points $x_1',x_2',x_3'$ and then the aligned regular simplex gradient (at the point $x_0$) can be computed in $\oo(n)$ operations via Theorem~\ref{T_Ongrad}:
\begin{eqnarray}\label{exp1:g2}
    g_2 = \mat{\phantom{-}0.049842074409398\\-0.007735568480143}.
\end{eqnarray}
Notice that $g_{2}$ is different from that given in \eqref{Grad_x0}; again, the signs and numbers do not match. In practice we do not have access to the true gradient so we are left to compare $g_{1}$ and $g_{2}$. Notice the sign of the first component $g_{1}$ is opposite from that of $g_{2}$ (so they point in different directions) and the numerical values of the components are also different.

In this situation it is beneficial to use the ideas from Section~\ref{S_Oh2} to improve the accuracy of the simplex gradient at $x_0$. To this end, from \eqref{Oh2grad} one can compute the $\oo(h^2)$ approximation to the true gradient:
\begin{eqnarray}\label{exp1:g12}
  g_{12} = 2g_2 - g_1 = \mat{0.195435033145664\\0.002024980112607}.
\end{eqnarray}
Clearly, $g_{12}$ is a very good approximation to the true gradient; the sign of $g_{12}$ matches that of $\nabla f(x_0)$, and the magnitude of the components aligns very well too, agreeing to 3 significant figures. Note that, because $h_1 = 10^{-3}$, one only expects 3 significant figures of accuracy.

We have repeated the experiment above for fixed $x_0$ and $\V$, but for varying values of $h_1$ (with the relation $h_2 = \tfrac12 h_1$ holding for each choice of $h_1$). The results are shown in Figure~\ref{fig:changingh}. The error is measured as the difference between the true gradient $\nabla f(x_0)$ stated in \eqref{Grad_x0} and '$g$', where $g$ is a notational placeholder for $g_1$ \eqref{E_Ong}, $g_2$ \eqref{E_Ong} or $g_{12}$ \eqref{Oh2gradgeneral} as appropriate. The purpose of this experiment is to show that, as $h_1$ shrinks, the error decreases linearly, as proven in Theorem~\ref{errorx0}. The upper bound on the error (again see Theorem~\ref{errorx0}) is $\tfrac12 L h \sqrt{n}$. We selected the value 2000 to approximate the Lipschitz constant, because $L \approx \|\nabla^2f(x_0) - \nabla^2f(x_1)\|_2/\|x_0 - x_1\|_2 = 1.0769\times 10^3 \leq 2000$, where $x_1$ was the simplex vertex computed for $h_1 = 10^{-3}$. Figure~\ref{fig:changingh} also shows that the gradient approximation $g_{12}$ --- formed by applying a Richardson's extrapolation strategy to $g_1$ and $g_2$ --- is very accurate.
\begin{figure}[h!]\centering
  \includegraphics[width=12cm]{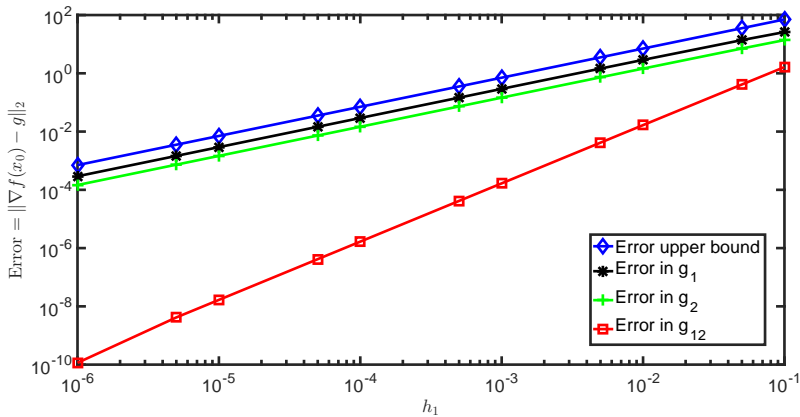}
  \label{fig:changingh}
  \caption{Plot showing the error in the gradient approximation as $h_1$ varies.}
\end{figure}

\subsection{High accuracy near the solution}\label{SS_exp2}
In this example we show how the techniques in Section~\ref{S_Oh2} can be used to hone in on a stationary point.
Suppose one wishes to compute the regular simplex gradient at the point
\begin{equation}\label{x0h2}
  x_0 = \mat{0.9\\0.81},
\end{equation}
which is close to the solution $x^* = \mat{1&1}.$ Using \eqref{E_rosenbrock_grad}, the analytic gradient at $x_0$ \eqref{x0h2} is
\begin{eqnarray}\label{Grad_x0h2}
  \nabla f(0.9,0.81) = \mat{-0.2000000000000000\\0}.
\end{eqnarray}
Now we construct the aligned regular simplex using the approach in Section~\ref{S_Simplex}. Here, $n=2$, $\alpha$, $\beta$ and $\gamma$ are the same as in \eqref{Numerical_abg}, and $\V$ is the same as in \eqref{NumericalVp}. We choose $h_1 = 10^{-6}$. The vertices of the simplex are computed as $x_j = x_0 + h_1v_j$ for $j = 1,2,3$ (see \eqref{vj}), and are the columns of
\begin{equation}\label{E_X}
  X_+ = \mat{x_1&x_2&x_3}=\mat{0.90000026&0.89999903&0.90000071\\
   0.80999903&0.81000026&0.81000071}.
\end{equation}
The simplex gradient is computed in $\oo(n)$ operations using Theorem~\ref{T_Ongrad} and is
\begin{eqnarray}\label{SimplexGradg1h2}
  g_1 = \mat{-0.200206828472801\\-0.000047729764447}.
\end{eqnarray}
The regular simplex gradient $g_1$ is a good approximation to the true gradient \eqref{Grad_x0h2}. The first component of \eqref{SimplexGradg1h2} has the same sign as the first component of \eqref{Grad_x0h2}, and they match to 3 significant figures. Also, the second component of \eqref{SimplexGradg1h2} is $\sim -5\times10^{-5}$, which while not exactly zero, is still small.

Now consider computing a second aligned regular simplex gradient, again at the point $x_0$, but now with $h_2 = -\tfrac12 h_1$, recall Remark~\ref{remarkOh2}(2). (We choose a negative multiple for demonstration purposes only.) The vertices of the simplex are computed as $x_j = x_0 + h_2v_j$ for $j = 1,2,3$ (see \eqref{vj}), and are the columns of
\begin{equation}\label{E_X}
  X_+ = \mat{x_1&x_2&x_3}=\mat{0.89999987&0.90000048&0.89999965\\
   0.81000048&0.80999987&0.80999965}.
\end{equation}
Using Theorem~\ref{T_Ongrad}, the regular simplex gradient is
\begin{eqnarray}\label{SimplexGradg1h22}
  g_2 = \mat{-0.199896585549141\\\phantom{-}0.000023864840841}.
\end{eqnarray}
Again, $g_2$ is a good approximation to the true gradient. The first components of \eqref{SimplexGradg1h22} and \eqref{Grad_x0h2} are very similar, and the second component of \eqref{SimplexGradg1h22} is also small. Notice that $g_1$ and $g_2$ are also similar, although the sign of the second component of $g_1$ is opposite that of $g_2$.
We can now use \eqref{Oh2grad} to combine $g_1$ and $g_2$ and obtain an $\oo(h^2)$ approximation to the true gradient:
\begin{eqnarray}
  g_{12} = \mat{-0.199999999857027\\
  -0.000000000027588}.
\end{eqnarray}
Clearly, $g_{12}$ is a very good approximation to $\nabla f(x_0)$. Notice that the approximation $g_{12}$ is now accurate to 10 decimal places.

These examples make it clear that obtaining a high accuracy aligned regular simplex gradient is cheap (once function evaluations have been computed). Each regular simplex gradient (i.e., $g_1$ and $g_2$) is obtained in $\oo(n)$ operations, and the $\oo(h^2)$ approximation $g_{12}$ is simply a weighted sum of $g_1$ and $g_2$, so it also costs $\oo(n)$.

In the above examples the  simplexes had the same centroid for each first order gradient calculation but this need not always be the case.
Sometimes the new simplex is obtained by shrinking (or expanding) the current
simplex keeping one of the vertices fixed and/or by rotating the current simplex about a vertex.  In such cases the formula \eqref{Oh2grad} can still be applied and gives a second order estimate at the vertex common to the two
simplexes used in the two first order estimates.  The so-called `centered
difference simplex gradient' \cite[p.115]{Kelley1999} is one such
example.
If the centroid of the simplex is not used it may also be convenient to replace the `arm-length' $h$ by the edge length $s$. These are simply related through the cosine rule ($s=\sqrt{2}\alpha h=h\sqrt{2+2/n}$).

We conclude this section with a schematic of the simplexes generated in each of these numerical experiments. The left plot in Figure~\ref{Fig_Simplex_numerics} relates to the experiment in Section~\ref{SS_exp1}, while the right plot relates to the experiment in Section~\ref{SS_exp2}. In the left plot in Figure~\ref{Fig_Simplex_numerics}, points $x_1,x_2,x_3$ (see \eqref{E_X}) represent vertices of the simplex with $h_1=10^{-3}$. Points $x_1',x_2',x_3'$ (see \eqref{E_Xp}) represent vertices of the simplex with $h_2=\frac12 h_1 =5\times 10^{-4}$. This choice of $h_2$ simply shrinks the regular simplex while maintaining the orientation of the original simplex. On the other hand, the right plot in Figure~\ref{Fig_Simplex_numerics} corresponds to the experiment in Section~\ref{SS_exp2}. In particular, points $x_1,x_2,x_3$ represent vertices of the simplex with $h_1=10^{-6}$. However points $x_1',x_2',x_3'$ represent vertices of the simplex with $h_2=-\frac12 h_1$. This choice of $h_2$ shrinks and also rotates the regular simplex (Because this is an aligned regular simplex, this is equivalent to rotating the simplex about $x_0$ by $180^{\circ}$).

\begin{figure}[h!]\centering
   \begin{tikzpicture}
    [scale=2.5]
    \draw[thick] (1.1,1.21) --(1.3588,0.2441);
  \draw[thick] (1.1,1.21) --(0.1341,1.4688);
  \draw[thick] (1.1,1.21) --(1.8071,1.9171);
  \draw [fill] (1.1,1.21) circle [radius=0.02];
  \draw [fill] (1.3588,0.2441) circle [radius=0.02];
  \draw [fill] (0.1341,1.4688) circle [radius=0.02];
  \draw [fill] (1.8071,1.9171) circle [radius=0.02];
  \draw [fill] (1.2294,0.7270) circle [radius=0.02];
  \draw [fill] (0.6170,1.3394) circle [radius=0.02];
  \draw [fill] (1.4536,1.5636) circle [radius=0.02];
  \node at (1.05,1.33) {$x_0$};
  \node at (1.45,0.15) {$x_1$};
  \node at (0,1.55) {$x_2$};
  \node at (1.9,2) {$x_3$};
  \node at (1.35,0.7) {$x_1'$};
  \node at (0.6,1.2) {$x_2'$};
  \node at (1.55,1.45) {$x_3'$};
\end{tikzpicture}
\hspace{1cm}
\begin{tikzpicture}
    [scale=2.5]
    \draw[thick] (1.1,1.21) --(1.3588,0.2441);
  \draw[thick] (1.1,1.21) --(0.1341,1.4688);
  \draw[thick] (1.1,1.21) --(1.8071,1.9171);
  \draw [fill] (1.1,1.21) circle [radius=0.02];
  \draw [fill] (1.3588,0.2441) circle [radius=0.02];
  \draw [fill] (0.1341,1.4688) circle [radius=0.02];
  \draw [fill] (1.8071,1.9171) circle [radius=0.02];
  \draw[thick] (1.1,1.21) --(0.9706,1.6930);
  \draw[thick] (1.1,1.21) --(1.5829,1.0806);
  \draw[thick] (1.1,1.21) --(0.7464,0.8565);
  \draw [fill] (0.9706,1.6930) circle [radius=0.02];
  \draw [fill] (1.5829,1.0806) circle [radius=0.02];
  \draw [fill] (0.7464,0.8565) circle [radius=0.02];
  \node at (0.95,1.33) {$x_0$};
  \node at (1.45,0.15) {$x_1$};
  \node at (0,1.55) {$x_2$};
  \node at (1.9,2) {$x_3$};
  \node at (1.05,1.75) {$x_1'$};
  \node at (1.73,1.10) {$x_2'$};
  \node at (0.6,0.8) {$x_3'$};
\end{tikzpicture}
\caption{A schematic of the simplexes generated in the numerical experiments. The left plot relates to the experiment in Section~\ref{SS_exp1}, while the right plot relates to the experiment in Section~\ref{SS_exp2}. }
\label{Fig_Simplex_numerics}
\end{figure}
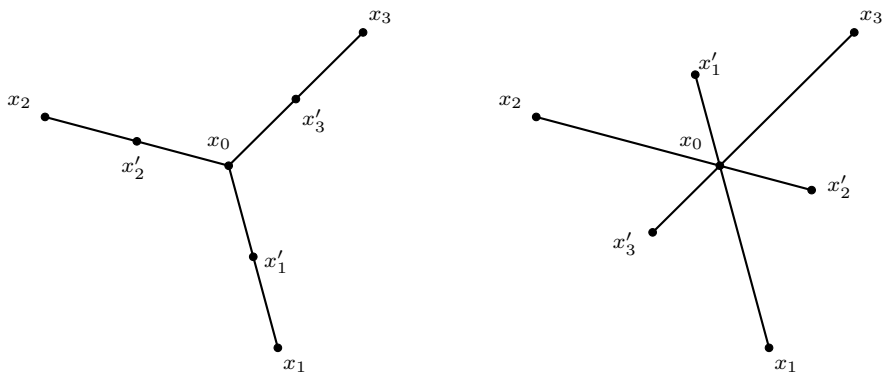

\section{Conclusion}\label{S_Conclusion}
In this work it was shown that a simplex gradient can be obtained efficiently, in terms of the linear algebra and memory costs, when the simplex is regular and appropriately aligned. A simplex gradient is the least-squares solution of a system of linear equations, which can have a computational cost of $\oo(n^3)$ for a general and unstructured system. However, due to the properties of the aligned regular simplex, the linear algebra of the least squares system simplifies, and the aligned regular simplex gradient is simply a weighted sum of the function values (measured at the vertices of the simplex) and a constant vector. Therefore, the computational cost of obtaining an aligned regular simplex gradient is only $\oo(n)$. Furthermore, the storage costs are low. Indeed, $\V$ need not be stored at all; the vertices of the aligned regular simplex can be constructed on-the-fly using only the centroid $x_0$ and radius $h$. Moreover, we have shown that, if the regular simplex is arbitrarily oriented, then the regular simplex gradient can be computed in at most $\oo(n^2)$.

Several extensions of this work were presented, including the easy generation of a simplex with integer coordinates when $n+1$ is a perfect square. We also showed that Richardson's extrapolation can be employed to obtain an $\oo(h^2)$ accuracy approximation to the true gradient from two regular simplex gradients.

\subsection{Future Work}

The main contribution of this work was to show that a regular simplex gradient can be determined efficiently in terms of the numerical linear algebra and storage costs. Simplex gradients are useful in a wide range of contexts and applications, including using the simplex gradient to determine a search direction, employing the simplex gradient in an algorithm termination condition, and determining when to shrink the mesh size in a grid based method. Future work includes embedding this inexpensive regular simplex gradient computation into an optimization routine to investigate how the regular simplex gradient calculation affects overall algorithm performance.

\begin{acknowledgements}
The authors thank Luis Vicente, and the anonymous referees for their helpful comments and suggestions, leading to improvements in an earlier version of this work.
\end{acknowledgements}

\bibliographystyle{spmpsci}
  \bibliography{COAP_DFO_bib}

\end{document}